\documentclass[a4paper,11pt,reqno]{amsart}

\usepackage{amssymb}
\usepackage{siunitx}
\usepackage{epsfig}
\usepackage{amsfonts,amsrefs}
\usepackage{amsmath}
\usepackage{euscript}
\usepackage{amscd}
\usepackage{amsthm}
\usepackage{enumitem}
\DeclareMathAlphabet{\mathpzc}{OT1}{pzc}{m}{it}
\usepackage{enumitem}
\usepackage{color}
\usepackage{mathtools}
\usepackage[hypertexnames=false, colorlinks, citecolor=red, linkcolor=blue, urlcolor=red]{hyperref}
\usepackage[utf8]{inputenc}
\usepackage[T1]{fontenc} 
\usepackage{marginnote}

\usepackage[normalem]{ulem}

\newcommand{\marginextend}[1]{ \addtolength{\oddsidemargin}{-#1}  \addtolength{\evensidemargin}{-#1}
	\addtolength{\textwidth}{#1}\addtolength{\textwidth}{#1}}
\newcommand{\updownextend}[1]{ \addtolength{\topmargin}{-#1}  \addtolength{\textheight}{#1}
	\addtolength{\textheight}{#1}}
\marginextend{1.5cm}
\updownextend{0cm}
\allowdisplaybreaks[4]

\usepackage{pst-node}
\usepackage{tikz-cd}
\usepackage{mathrsfs}
\usepackage[most]{tcolorbox}

\DeclareFontFamily{OT1}{pzc}{}
\DeclareFontShape{OT1}{pzc}{m}{it}{<-> s * [1.10] pzcmi7t}{}
\DeclareMathAlphabet{\mathpzc}{OT1}{pzc}{m}{it}

\DeclareSymbolFont{SY}{U}{psy}{m}{n}
\DeclareMathSymbol{\emptyset}{\mathord}{SY}{'306}

\theoremstyle{plain}

\newtheorem{thm}{Theorem}[section]
\newtheorem*{thm*}{Theorem}
\newtheorem{cor}[thm]{Corollary}
\newtheorem{lem}[thm]{Lemma}
\newtheorem{prop}[thm]{Proposition}
\newtheorem{defn}[thm]{Definition}
\newtheorem{rem}[thm]{Remark}

\newtheoremstyle{mainthmstyle}%
  {}{}
  {\itshape}
  {}
  {\bfseries}
  {}
  {0pt}
  {\thmname{#1}. \thmnote{#3}} 

\theoremstyle{mainthmstyle}

\newtheoremstyle{named}{}{}{\itshape}{}{\bfseries}{.}{.5em}{#1 \thmnote{#3}}
\theoremstyle{named}

\newenvironment{maintheorem}
{\par\noindent\textbf{Main Theorem.}\itshape}
{\par}

\tcolorboxenvironment{maintheorem}{
  colback=white,
  colframe=black,
  boxrule=0.8pt
}

\numberwithin{equation}{section}

\def\beq{\begin{eqnarray}}
	\def\eeq{\end{eqnarray}}
\def\beqa{\begin{eqnarray*}}
	\def\eeqa{\end{eqnarray*}}

\def\M{\boldsymbol M}

\newcommand{\be}{\begin{equation}}
	\newcommand{\ee}{\end{equation}}
\newcommand{\bea}{\begin{eqnarray}}
	\newcommand{\eea}{\end{eqnarray}}
\newcommand{\Bea}{\begin{eqnarray*}}
	\newcommand{\Eea}{\end{eqnarray*}}

\newcommand{\clipnorm}{$_{\text{\rm c}}$Lip-norm}
\newcommand{\clipnorms}{$_{\text{\rm c}}$Lip-norms}

\newcounter{cnt1}
\newcounter{cnt2}
\newcounter{cnt3}
\newcommand{\blr}{\begin{list}{$($\roman{cnt1}$)$}
		{\usecounter{cnt1} \setlength{\topsep}{0pt}
			\setlength{\itemsep}{0pt}}}
	\newcommand{\bla}{\begin{list}{$($\alph{cnt2}$)$}
			{\usecounter{cnt2} \setlength{\topsep}{0pt}
				\setlength{\itemsep}{0pt}}}
		\newcommand{\bln}{\begin{list}{$($\arabic{cnt3}$)$}
				{\usecounter{cnt3} \setlength{\topsep}{0pt}
					\setlength{\itemsep}{0pt}}}
			\newcommand{\el}{\end{list}}
		
		\vspace{5mm}

		\title{Metric Dimension and Product Entropy of Group $C^{\ast}$-algebras}
		\author[A. Chattopadhyay]{Arnab Chattopadhyay}
		\author[S. Joardar]{Soumalya Joardar}

        \address[A. Chattopadhyay]{Indian Institute of Science Education And Research Kolkata, Mohanpur 741246, Nadia, West Bengal, India} \email{ac23rs002@iiserkol.ac.in}

        \address[S. Joardar]{Indian Institute of Science Education And Research Kolkata, Mohanpur 741246, Nadia, West Bengal, India} \email{soumalya@iiserkol.ac.in}
		
\begin{document}

\begin{abstract}
 We consider reduced group $C^{\ast}$-algebras of finitely generated discrete groups metrized by seminorms obtained from word length functions as in \cite{Christensen}. We study the metric dimension of such $C^{\ast}$-algebras as defined in \cite{Kerr}. We also study the product entropy of the automorphisms of group $C^{\ast}$-algebras induced by the automorphisms of the underlying groups. We get a lower bound and an upper bound of the product entropy of an automorphism in terms of the classical group theoretic algebraic and geometric entropy of the automorphisms, provided the group has polynomial growth property. For groups with exponential growth, we show that the metric dimension of the group $C^{\ast}$-algebras is generically $+\infty$. 
\end{abstract}
\maketitle
\section{Introduction} Compact quantum metric spaces are noncommutative generalizations of classical compact topological spaces. It was initially proposed by A. Connes within his spectral triple framework (\cite{Connes1}). Then M. Rieffel came up with a formal definition using what he called Lip-norms on $C^{\ast}$-algebras (\cite{Rieffel1}). Recall that a densely defined seminorm on a unital $C^{\ast}$-algebra $A$ with one dimensional kernel produces a metric on the state space $S(A)$ (\cite{Rieffel1}). Then such a seminorm is called a Lip-norm if the metric induces the weak$^\ast$-topology on the state space.  A rich vein of examples comes from the reduced group $C^{\ast}$-algebras of discrete groups. In \cite{Rieffel2}, it was shown that $C^{\ast}_{r}(\mathbb{Z}^{d})$ and the hyperbolic group $C^{\ast}$-algebras are examples of compact quantum metric spaces where the Lip-norms come from word length functions. Then in \cite{Christensen}, Christensen and Antonescu constructed a sequence of eventual Lip-norms induced from a proper length function on the reduced $C^{\ast}$-algebras of discrete groups which enjoy the rapid decay property.\\
\indent One of the key concepts of dynamical systems is that of entropy. It roughly measures the chaos in a system. Many attempts have been made to generalize the notion of classical entropy to the realm of operator algebras. Some such notions are Brown-Voiculescu entropy (\cites{Voiculescu, Brown}), Connes-Stormer entropy (\cites{Connesstormer}) and its generalization Connes-Narnhofer-Thirring entropy (\cite{CNT}), Sauvageot-Thouvenot entropy (\cite{Sauvageot}). With the introduction of compact quantum metric spaces, it is natural to consider metric dimension and entropy of those spaces that would generalize the classical metric dimension (for example the Kolmogorov dimension of a compact metric space) and the classical topological entropy. It was considered in the paper \cite{Kerr} in which David Kerr proposed the notion of metric dimension and product entropy of compact quantum metric spaces. They naturally become bi-Lipschitz $C^{\ast}$-algebraic invariants. In that paper, after defining the metric dimension and the product entropy, Kerr goes on to prove that the metric dimension of $C(X)$ coincides with the Kolmogorov dimension of $X$ under the canonical choice of Lip-norm induced by the metric $d$ of $X$. In \cite{Kerr}, some explicit computations of the product entropy of some automorphisms of the non-commutative torus and $M_{p^{\infty}}$ have been done. Surprisingly, after \cite{Kerr}, there is hardly any follow up. It is quite natural to see some more computations of metric dimension as well as product entropy of some other compact quantum metric spaces. As already mentioned, if $\Gamma$ is some discrete group having the rapid decay property, its reduced group $C^{\ast}$-algebra $C^{\ast}_{r}(\Gamma)$ becomes a compact quantum metric space. In this paper, we show that if $\Gamma$ has polynomial growth, then $C^{\ast}_{r}(\Gamma)$ has finite metric dimension. We slightly modify the definition of the product entropy due to Kerr to make it independent of the CQMS structure on $C^{\ast}_{r}(\Gamma)$. Given an automorphism $\phi$ of $C^{\ast}_{r}(\Gamma)$, the entropy of $\phi$ in our sense roughly measures the growth of complexity of finite subsets $\Omega$ of the unit ball in the group algebra $C_{c}(\Gamma)$ under the iteration of $\phi$. It coincides with the definition given by Kerr for the choice of CQMS structures on $C^{\ast}_{r}(\Gamma)$ considered in this paper. The consideration of unit ball in the group algebra might seem restrictive at first glance. But Theorem \ref{entpZ} shows that consideration of unit ball of the group algebra is good enough to recover the classical entropy of automorphisms of the abelian group $\mathbb{Z}^d$ for any $d$. Recall that any automorphism of the group $\Gamma$ produces an automorphism of $C^{\ast}_{r}(\Gamma)$. We study the product entropy of such automorphisms. For a large class of such automorphisms (like the inner automorphisms, see Theorem \ref{innerzero}), the entropy turns out to be zero, provided $C^{\ast}_{r}(\Gamma)$ has finite metric dimension. The case of abelian groups reduces to the case of the tori via the Fourier transform. For a general discrete group $\Gamma$ and an automorphism $\phi$ of $\Gamma$, we show that the product entropy is always greater than or equal to the classical algebraic entropy (in the sense of \cite{MR2491886}) of $\phi$. The classical algebraic entropy of an automorphism of a finitely generated discrete group is well-known, particularly for groups with polynomial growth. By Gromov's Theorem \cite{MR623534}, finitely-generated groups of polynomial growth are virtually nilpotent. Although, for torsionless virtually nilpotent groups, we have partial information about the algebraic entropy due to the lack of additivity theorem, it is quite useful to obtain the lower bound of the product entropy (see the discussion after Remark \ref{Lips}). Torsion nilpotent groups do admit such additivity property (see \cite{shlossberg2026addition}*{Theorem 4.5}) and algebraic entropy can be computed explicitly. But as we consider finitely generated discrete groups, torsion groups are essentially finite and therefore does not have any kind of growth. Therefore, we only consider torsionless groups in this paper. To get an upper bound of the product entropy, the finiteness of the metric dimension becomes crucial. For groups with polynomial growth, we get upper bound combining the growth exponent of the group and the geometric entropy (see Definition \ref{geoment}) of the auotomorphism. The geometric entropy of an automorphism of a virtually nilpotent group is completely known. Therefore, both the bounds are `nice' in the sense that for groups of polynomial growth, we have fair bit of knowledge of both the upper and lower bounds from the study of classical entropy of endomorphisms of the groups. This helps us to produce finite non-zero product entropy of a large class of automorphisms on groups of polynomial growth.  This is quite interesting, as Kerr notes in \cite{Kerr2} that it is in general a challenge to produce non-zero finite entropy for non-commutative $C^{\ast}$-algebraic dynamics. This belief is further reinforced in a recent work (\cite{new2025}) in which it is shown that infinite entropy occurs generically for automorphisms of certain classifiable $C^{\ast}$-algebras. In case of noncommutative torus, however, the toral automorphisms do admit non-zero finite entropy. But in the realm of group $C^{\ast}$-algebras, the occurrence of finite non-zero entropy is new. Calculating the Brown-Voiculescu entropy for examples such as $C^{\ast}_{r}(H_3 (\mathbb Z))$ seems to be difficult. But here we get nice bounds for the product entropy for automorphisms of $C^{\ast}_{r}(\Gamma)$ for any group $\Gamma$ with polynomial growth. For groups with exponential growth admitting the rapid decay property, we show that the metric dimension of the associated reduced group $C^{\ast}$-algebras is always $+\infty$ (see Theorem \ref{expoinfty}). The product entropy of the trivial automorphism of a finitely generated discrete group with exponential growth is $+\infty$. We believe that the product entropy of any automorphism of groups with exponential growth is $+\infty$. It is worthwhile to mention that Kerr considers entropy of bi-Lipschitz automorphisms. But to define the entropy, the bi-Lipschitzness is not really needed. In fact, the automorphisms considered in this paper are rarely Lipschitz.  We also note that a lot of computations can be done without the exact Leibnitz property of the Lip-norm as demanded by Kerr. In fact, the natural Lip-norms on $C^{\ast}_{r}(\Gamma)$ do not admit the Leibnitz property. But they satisfy a `higher order' Leibnitz-type property (look at Lemma \ref{higherleibnitz}) which aids our computations. 
\section{Preliminaries}
\subsection{Compact quantum metric spaces}
We begin this subsection with the definition of compact quantum metric space (CQMS for short). Compact quantum metric spaces can be defined for order unit spaces. But we shall deal with only unital $C^{\ast}$-algebras in this paper and therefore we define CQMS for unital $C^{\ast}$-algebras. We refer the reader to \cite{Rieffel1} for more details. Let $A$ be a unital $C^{\ast}$-algebra with unit $1_{A}$. A seminorm $L:A\rightarrow[0,+\infty]$ on $A$ possibly taking the value $+\infty$ is said to separate the state space $S(A)$ if given $\phi,\psi\in S(A)$ there is an element $a\in A$ such that $L(a)<+\infty$ and $\phi(a)\neq\psi(a)$. Given a separating seminorm $L$ on $A$, we denote the set $\{a\in A:L(a)<r\}$ by $\mathcal{L}_{r}$. The set $\{a\in A:L(a)<+\infty\}$ will be denoted by $\mathcal{L}$ and will be called the domain of $L$. One can define a metric on the state space $S(A)$ by the following formula:
\begin{displaymath}
    d_{L}(\phi,\psi)=\sup\limits_{a\in\mathcal{L}_{1}}\lvert\phi(a)-\psi(a)\rvert,
\end{displaymath}
for $\phi,\psi\in S(A)$. \\
A \clipnorm\ on $A$ is a separating seminorm $L:A\rightarrow[0,+\infty]$ such that
\begin{itemize}
\item $L(a^{\ast})=L(a)$ for all $a\in A$.
    \item ${\textrm Ker}(L)=\mathbb{C}1_{A}$.
    \item $d_{L}$ metrizes the weak $\ast$-topology on $S(A)$.
\end{itemize}
\begin{defn}\label{CQMSdefn}
    A compact quantum metric space is a pair $(A,L)$ such that $A$ is a unital $C^{\ast}$-algebra and $L$ is a \clipnorm\ on $A$.
\end{defn}
\begin{rem}
    We have taken the definition of a \clipnorm\ as given in Definition 2.3 of \cite{Kerr}. Although M. Rieffel originally considered Lip-norms for his definition of compact quantum metric spaces, we stick to the \clipnorms\ in the sense of \cite{Kerr}. In our case as $\mathcal{L}$ will always be dense in $A$, $L$ will separate points in $S(A)$ in the sense of \cite{Kerr}. 
\end{rem}
For examples of CQMS the reader is referred to \cites{Rieffel2,Rieffel-1998-Metrics-on-state-from-action-of-cmpt-gp}. We are only going to discuss the following class of examples needed for our purpose.\vspace{0.05in}\\
 \textbf {Group $C^{\ast}$-algebras}: Let $\Gamma$ be a finitely generated discrete group equipped with a length function $\ell$. $\Gamma$ is said to have
 \begin{itemize}
     \item polynomial growth with growth exponent $r \geq 1$ if there exists $C > 0$ such that for all $n \in \mathbb N$ $$\left \lvert \left \{g \in \Gamma\ :\ \ell (g) \leq n \right \} \right \rvert \leq C (1 + n)^{r};$$
     \item rapid decay property with decay exponent $r > 0$ if there exists $C > 0$ such that for all $f \in C_c (\Gamma)$ $$\|f\|_{\mathrm {red}} \leq C \left (\sum\limits_{g \in \Gamma} \left \lvert f (g) \right \rvert^2 \left (1 + \ell (g) \right )^{2 r} \right )^{\frac {1} {2}}.$$
 \end{itemize}
 It is well known that polynomial growth implies the rapid decay property. Now, let $\Gamma$ be a finitely generated discrete group equipped with a length function such that $\Gamma$ has the rapid decay property. Then $C^{\ast}_{r}(\Gamma)$ is the reduced group $C^{\ast}$-algebra faithfully represented on the Hilbert space $\ell^{2}(\Gamma)$. We shall denote the representation by $\lambda$. Recall the dense $\ast$-subalgebra $C_{c}(\Gamma):=\{\sum_{g\in F\subseteq\Gamma}c_{g}\delta_{g}:c_{g}\in\mathbb{C}\}$, where $F$ is some finite subset of the group $\Gamma$. By definition, for any $f\in C_{c}(\Gamma)$, $\lvert\lvert f\rvert\rvert_{\textrm{red}}=\lvert\lvert\lambda(f)\rvert\rvert_{\textrm{op}}$. Then using the length function, one can define the following unbounded self adjoint operator on $\ell^{2}(\Gamma)$ with dense domain $C_{c}(\Gamma)$:
 \begin{displaymath}
     M_{\ell}\left(\sum\limits_{g\in F\subseteq\Gamma}c_{g}\delta_{g} \right )=\sum\limits_{g\in F\subseteq\Gamma}\ell(g)\ c_{g}\delta_{g}.
 \end{displaymath}
 Then it is well known that $[M_{\ell},\lambda(f)]\in B(\ell^{2}(\Gamma))$ for all $f\in C_{c}(\Gamma)$. In fact, one can take repeated commutators and produce a sequence of bounded operators $\Delta_{\ell}^{k}(f)$ for a fixed $f\in C_{c}(\Gamma)$ i.e. $\Delta^{k}_{\ell}(f)$ is given inductively by
 \begin{displaymath}
    \Delta_{\ell}^{1}(f):=[M_{\ell}, \lambda(f)], \  \Delta^{k}_{\ell}(f):=[M_{\ell},\Delta^{k-1}(f)].
 \end{displaymath}
 Note that for $f_{1},f_{2}\in C_{c}(\Gamma)$, $\lambda(f_{1})\lambda(f_{2})=\lambda(f_{1}\ast f_{2})$, where $\ast$ is the convolution product of the group algebra. We shall denote $f_{1}f_{2}$ to denote the convolution product of two elements $f_{1},f_{2}\in C_{c}(\Gamma)$. Then for $f_{1},f_{2}\in C_{c}(\Gamma)$, $f_{1}f_{2}\in C_{c}(\Gamma)$ so that $\Delta^{k}_{\ell}(f_{1}f_{2}\ldots f_{n})\in B(\ell^{2}(\Gamma))$ for any $f_{1},f_{2},\ldots,f_{n}\in C_{c}(\Gamma)$.
 The following lemma says that $\Delta_{\ell}^{k}$ satisfies a higher order Leibnitz property. We have decided to keep the proof, which is a lengthy computation, in the Appendix section. This `higher order Leibnitz property' is going to be very useful for computations.
\begin{lem} \label{higherleibnitz}
For all $k, n \in \mathbb N$ we have \Bea \Delta_{\ell}^{k} \left (f_{1}f_{2} \cdots f_{n} \right ) & = & \sum\limits_{r_{1} + r_{2} + \cdots + r_{n} = k} \dbinom {k} {r_{1}\ r_{2}\ \cdots\ r_{n}} \prod\limits_{i = 1}^{n} \Delta_{\ell}^{r_{i}} \left (f_{i} \right ), \Eea where $\dbinom {k} {r_{1}\ r_{2}\ \cdots\ r_{n}} = \dfrac {k!} {r_{1}!\ r_{2}!\ \cdots\ r_{n}!}$ and $f_{i} \in C_{c}(\Gamma)$ for $i = 1, 2, \cdots, n.$ 
\end{lem}
One can define the following sequence of separating (as $L_{\ell}^{k}$'s are densely defined) seminorms on $C^{\ast}_{r}(\Gamma)$:
$$L_{\ell}^k (f) : = \begin{cases} \left \|\Delta^{k} (f) \right \|_{\mathrm {op}} \quad f \in C_c (\Gamma), \\ 0 \quad \quad \hspace{13.2mm} \mathrm {otherwise}.\end{cases}$$
 We have the following corollary of the Lemma \ref{higherleibnitz}:
\begin{cor} \label{LeibnitzL}
For any $k, n \in \mathbb N$ we have \Bea L_{\ell}^{k} \left (f_{1}f_{2}\cdots f_{n} \right ) & \leq & \sum\limits_{r_{1} + r_{2} + \cdots + r_{n} = k} \dbinom {k} {r_{1}\ r_{2}\ \cdots\ r_{n}} \prod\limits_{i = 1}^{n} L_{\ell}^{r_{i}} \left (f_{i} \right ), \Eea where $\dbinom {k} {r_{1}\ r_{2}\ \cdots\ r_{n}} = \dfrac {k!} {r_{1}!\ r_{2}!\ \cdots\ r_{n}!}$ and $f_{i} \in C_{c} (\Gamma)$ for $i = 1, 2, \cdots, n.$
\end{cor}

\begin{proof}
The proof follows directly from Lemma \ref{higherleibnitz} by taking the operator norm in both sides of the inequality followed by using triangle inequality and sub-multiplicativity of the norm.
\end{proof}
Note that $L^{k}_{\ell}$'s are adjoint invariant and ${\textrm Ker}(L^{k}_{\ell})=\mathbb{C}1_{A}$ for all $k$. It has already been noted that $L^{k}_{\ell}$'s separate points. With these observations, we have the following crucial theorem:
\begin{thm}\label{CQMSmain}(\cite{Christensen}*{Theorem 2.6})
    Let $\Gamma$ be a discrete group with the rapid decay property with decay exponent $r$ and $\ell$ be a proper length function on $\Gamma.$ Then for any $k>r$, $d_{L^{k}_{\ell}}$ (the metric induced by $L_{\ell}^{k}$) metrizes the weak$^{\ast}$-topology of $S(C^{\ast}_{r}(\Gamma))$ and therefore, $(C_{r}^{\ast}(\Gamma),L_{\ell}^{k})$ is a compact quantum metric space.
\end{thm}
We shall use two crucial estimates in this paper which we are going to recall now. The reader is referred to \cite{Christensen}:\\
As before, let $\Gamma$ be a finitely generated discrete group with the word length function such that $\Gamma$ has the rapid decay property with decay exponent $r$. Fix a natural number $k>r$. 
Let $f=\sum\limits_{g\in F} a_{g}\delta_{g}\in C_{c}(\Gamma)$. Then using the proof of Theorem \ref{CQMSmain} we have
\begin{equation} \label{eqn 2.0}
\left \|\sum\limits_{\ell (g) \geq n} a_{g} \delta_{g} \right \|_{\mathrm {red}} \leq C 2^{r} n^{r - k}, 
\end{equation} for all natural number $n \in \mathbb N,$ where $C > 0$ is the constant of rapid decay;
\begin{equation} \label{eqn 2.1}
L_{\ell}^{k} (f) \leq \left \|M_{\ell}^{k} (f) \right \|_{1} = \sum\limits_{g \in G} \ell (g)^{k} \left \lvert a_{g} \right \rvert,
\end{equation} and
\begin{equation}\label{Lbound}L^{k}_{\ell}(f)\geq \left \|\Delta^{k}(f)(\delta_{e}) \right \|_{2}= \left (\sum_{g\in G}\ell(g)^{2k}\lvert a_{g}\rvert^{2}\right)^{\frac{1}{2}},\end{equation}
where $\| \cdot \|_{2}$ is the $\ell^{2}$-norm in $\ell^{2}(\Gamma)$.
\begin{lem}\label{deltag}
Let $\Gamma$ be a discrete group and $\ell$ be length function on $\Gamma.$ Then for any $k \in \mathbb N$ and for all $g \in \Gamma$ $$L_{\ell}^{k} \left (\delta_{g} \right ) = \left (\ell (g) \right )^{k}.$$
\end{lem}
\begin{proof}
The proof directly follows from the inequalities \eqref {eqn 2.1} and \eqref {Lbound}.
\end{proof}
\subsection{Metric dimension and the product entropy} We start this subsection by recalling a few notations and definitions from Section 3 of \cite{Kerr}. For a normed linear space $(X,\| \cdot \|)$ (either a $C^{\ast}$-algebra or a Hilbert space for us), $\mathcal{F}(X)$ will denote the collection of all finite dimensional subspaces of $X$. If $Y,Z$ are subsets of $X$, then for $\delta>0$, the notation $Y\subseteq_{\delta} Z$ will mean that for every $y\in Y$, there is some $z\in Z$ such that $\lvert\lvert y-z\rvert\rvert<\delta$. For any subset $Y\subset X$, $D(Y,\delta)=\inf\{{\textrm{dim}}(Z):Z\in\mathcal{F}(X), Y\subseteq_{\delta}Z\}$, where $\textrm{dim}(Z)$ is the vector space dimension of $Z$. Then it is easy to see that if $Y_{1}\subseteq Y_{2}$, $D(Y_{1},\delta)\leq D(Y_{2},\delta)$ for any $\delta>0$. Now let $(A,L)$ be a compact quantum metric space on a unital $C^{\ast}$-algebra in the sense of Definition \ref{CQMSdefn}. Recall the notation $\mathcal{L}_{r}$. We have the following lemma from \cite{Kerr} :
\begin{lem} (\cite{Kerr}*{Proposition 3.2})
    $D(\mathcal{L}_{1},\delta)$ is finite for all $\delta>0$.
\end{lem}
We record the following theorem due to Voiculescu which is extremely important to get lower bounds of metric dimension and entropy.
\begin{thm}\label{Voiculescu}(\cite{Voiculescu}*{Lemma 7.8}) Let $X$ be an orthonormal set in a Hilbert space $H$. Then for any $\delta>0$, $D(X,\delta)\geq(1-\delta^{2}){\mathrm {Card}}(X)$. 
\end{thm}
We will be interested in $C^{\ast}$-algebras. The above theorem connects with the $C^{\ast}$-algebraic approximations due to the result that we are going to state below:\\
Let $A$ be a unital $C^{\ast}$-algebra. For any state $\sigma\in S(A)$ (not necessarily faithful), denote the GNS space by $H_{\sigma}$, the corresponding cyclic vector by $\xi_{\sigma}$ and the representation by $\pi_{\sigma}$. Let $\Omega$ be a finite subset of $A$ such that $\Omega_{\sigma}:=\{\pi_{\sigma}(a)\xi_{\sigma}:a\in\Omega\}$ is an orthonormal set in $H_{\sigma}$. Then we have the following lemma:
\begin{lem}\label{VoiculescuCalg} (\cite{Kerr}*{Lemma 5.13})
 For any $\delta>0$, $D(\Omega,\delta)\geq D(\Omega_{\sigma},\delta)$.  
\end{lem}

\begin{rem}
Recall that the left regular representation $(\ell^2(\Gamma),\lambda,\delta_e)$ is precisely the GNS representation with respect to the canonical tracial state on $C_r^*(\Gamma)$ with cyclic vector $\delta_e.$ Therefore, by Lemma \ref{VoiculescuCalg}, for every finite subset $\Omega \subseteq C_r^*(\Gamma)$ and every $\delta>0$,
\[
D\bigl(\lambda(\Omega)\delta_e,\delta\bigr)\leq D(\Omega,\delta).
\]
We shall use this observation frequently in what follows.
\end{rem}

Now we are ready to recall the definition of the metric dimension of a compact quantum metric space $(A,L)$ on a unital $C^{\ast}$-algebra $A$.
\begin{defn}
    The metric dimension of a CQMS $(A,L)$ is defined to be
    \begin{displaymath}
        \mathrm{Mdim}_{L}(A):=\limsup\limits_{\delta\to 0^{+}} \frac{\log \ D(\mathcal{L}_{1},\delta)}{\log \ \delta^{-1}}
    \end{displaymath}
\end{defn}
Now we shall define the product entropy of an automorphism $\alpha$ of a group $C^{\ast}$-algebra $C^{\ast}_{r}(\Gamma)$ for a finitely generated discrete group $\Gamma$. To that end, let $Pf(X)$ denote the collection of all finite subsets of a set $X$. For any subsets $X_{1},X_{2},\ldots,X_{n}$, $\prod\limits_{i=1}^{n} X_{i}$ will denote the set $\{a_{1}a_{2}\ldots a_{n}\in A:a_{i}\in X_{i}\}$. We denote the unit ball of a $C^{\ast}$-algebra $A$ by $A_{1}$. For any subspace $\mathcal{A}\subset A$, we denote the set $\mathcal{A}\cap A_{1}$ by $\mathcal{A}_1$.
\begin{defn}
    Let $\Gamma$ be a finitely generated discrete group with an automorphism $\alpha$ viewed as an automorphism on $C^{\ast}_{r}(\Gamma)$. For $\Omega\in Pf(C_{c}(\Gamma)_{1})$ and $\delta>0$, define
    \begin{align*}
        &&\mathrm{Entp} (\alpha,\Omega,\delta) & =  \limsup\limits_{n\to\infty}\frac{1}{n}\log \ D \left (\Big(\prod_{i=0}^{n-1}\alpha^{i}(\Omega)\Big),\delta \right ),\\
        &&\mathrm{Entp}(\alpha,\Omega) & = \sup_{\delta>0}\mathrm{Entp}(\alpha,\Omega,\delta),\\
        &&\mathrm{Entp}(\alpha) & =  \sup_{\Omega\in Pf(C_{c}(\Gamma)_1)}\mathrm{Entp}(\alpha,\Omega).
    \end{align*}
\end{defn}
\begin{rem}\label{Kerrcomparison}
    Note that our definition of product entropy agrees with that given in \cite{Kerr}(Definition 5.2) for the CQMS structure on $C^{\ast}_{r}(\Gamma)$ given by $L_{\ell}^{k}$ for any $k$ i.e. for any automorphism $\alpha$, 
    \begin{displaymath}
        \mathrm{Entp}(\alpha)=\mathrm{Entp}_{L^{k}_{\ell}}(\alpha)
    \end{displaymath}
    for any appropriate $k$, where $\mathrm{Entp}_{L_{\ell}^k}$ stands for the product entropy as defined by Kerr for the CQMS structure $(C^{\ast}_{r}(\Gamma),L_{\ell}^{k})$. 
\end{rem}
\section{Main results}
Let $\Gamma$ be a finitely generated discrete group. Note that for meaningful discussion on metric dimension and entropy on $C^{\ast}_{r}(\Gamma)$, $\Gamma$ has to be an infinite group. We denote the word length function on $\Gamma$ by $\ell$. We also assume that $\Gamma$ has the rapid decay property with some exponent $r$ so that $(C^{\ast}_{r}(\Gamma),L^{k}_{\ell})$ is a CQMS for all $k>r$ (Theorem \ref{CQMSmain}). 


\vspace{2mm}
Throughout the remainder of this paper, we work exclusively with finitely generated discrete groups. Unless explicitly specified otherwise, the length function to be considered always refers to the word length function with respect to a fixed generating set of the underlying group.

\vspace{2mm}

Now we prove that if a group $\Gamma$ has polynomial growth, then $C^{\ast}_{r}(\Gamma)$ has finite metric dimension.
\begin{thm}\label{finiteMdim} Let $\Gamma$ be a finitely generated discrete group with polynomial growth of exponent $r$ equipped with the word length function $\ell$ relative to a finite generating set $S.$ Then for any $k>\frac {r} {2},$ $(C^{\ast}_{r}(\Gamma),L^{k}_{\ell})$ is a CQMS and
\begin{displaymath}
    \frac {1} {k} \leq \mathrm {Mdim}_{L_{\ell}^{k}} \left (C_{r}^{\ast} (\Gamma) \right ) \leq \frac {2 r} {2 k - r}.
\end{displaymath}    
\end{thm}
\begin{proof}
 First of all it follows from Theorem \ref{CQMSmain} and \cite{nica2010degree}*{Proposition 3.2} that $\left (C_r^{\ast} (\Gamma), L_{\ell}^{k} \right )$ is a CQMS for $k > \frac {r} {2}.$ Now consider the set $\mathcal L_{1} : = \left \{f \in C_{c} (\Gamma)\ :\ L_{\ell}^{k} (f) \leq 1 \right \}.$ Let us fix some $\varepsilon > 0.$ Then for any $\delta > 0$ there exists $X \in \mathcal {F} \left (C_r^{\ast} (\Gamma) \right ),$ a finite dimensional subspace with $\mathcal L_{1} \subseteq_{\delta} X,$ i.e., $\mathcal L_{1}$ is approximately contained in $X$ within $\delta,$ such that $$\dim (X) \leq D \left (\mathcal L_{1}, \delta \right ) + \varepsilon.$$ Let $X : = \mathrm {span} \left \{x_{1}, x_{2}, \cdots, x_{n} \right \}.$ Define a finite dimensional subspace $Y \in \mathcal F \left (\ell^{2} (\Gamma) \right )$ by $Y : = \mathrm {span} \left \{\lambda \left (x_{1} \right ) \delta_{e}, \lambda \left (x_{2} \right ) \delta_{e}, \cdots, \lambda \left (x_{n} \right ) \delta_{e} \right \},$ where $\lambda$ is the left regular representation of $C_{r}^{\ast} (\Gamma)$ on $\ell^{2} (\Gamma)$ and $e$ is the identity element of $\Gamma$ so that $\delta_{e}$ is the cyclic vector for the representation. Now given any $f \in \mathcal L_{1}$ get hold of $x \in X$ such that $\|f - x\|_{\mathrm {red}} < \delta.$ Then we have $$\left \|\lambda (f) \delta_{e} - \lambda (x) \delta_{e} \right \|_{\ell^{2} (G)} \leq \left \|\lambda (f) - \lambda (x) \right \|_{\mathrm {op}} = \left \|\lambda (f - x) \right \|_{\mathrm {op}} = \left \|f - x \right \|_{\mathrm {red}} < \delta.$$ But this shows that $\lambda(\mathcal L_{1})\delta_{e} \subseteq_{\delta} Y$ in $\ell^{2} (\Gamma).$ Therefore, $$D \left (\lambda(\mathcal L_{1})\delta_{e}, \delta \right ) \leq \dim (Y) \leq n = \dim (X) \leq D \left (\mathcal L_{1}, \delta \right ) + \varepsilon.$$ Letting $\varepsilon \rightarrow 0^{+},$ it follows that $$D \left (\lambda(\mathcal L_{1})\delta_{e}, \delta \right ) \leq D \left (\mathcal L_{1}, \delta \right ).$$
For any $n \in \mathbb N,$ let $\gamma_S (n) : = \left \lvert \{g \in G\ :\ell (g) \leq n \} \right \rvert.$ Then there exists $C > 0$ and $n_0 \in \mathbb N$ such that $$\gamma_S (n) \geq C n,$$ for all $n \geq n_0$ \cite{Dikranjan2013TopologicalEA}*{Fact 5.3.3.}. Fix any $0 < C^{\prime} < C$ and choose $\delta_0 > 0$ in such a way that $\left (\delta_{0}^{-1} \right )^{\frac {1} {k}} > \max \left \{n_0, \frac {C} {C - C^{\prime}} \right \}.$
For $0 < \delta \leq \delta_0,$ let us consider the set $$U_{\delta} : = \left \{\delta_{g} \ :\ \ell (g) \leq \left (\delta^{-1} \right )^{\frac {1} {k}} \right \}.$$ 
Since $L_{\ell}^{k} \left (\delta_{g} \right ) = \ell (g)^{k}$ for all $g \in \Gamma,$ it follows that $\delta U_{\delta} \subseteq \mathcal L_{1}.$ 
Note that $\lambda \left (U_{\delta} \right ) \delta_e$ is an orthonormal set in $\ell^2 (\Gamma).$ So by Voiculescu's theorem it follows that 
\Bea 
D \left (\lambda(U_{\delta})\delta_{e}, 2^{-1} \right ) 
& \geq & \frac {3} {4} \left \lvert U_{\delta} \right \rvert \\ 
& \geq & \frac {3} {4} \gamma_S \left (\left \lfloor \left (\delta^{-1} \right )^{\frac {1} {k}} \right \rfloor \right ) \\ 
& \geq & \frac {3} {4} C \left \lfloor \left (\delta^{-1} \right )^{\frac {1} {k}} \right \rfloor \\ & \geq & \frac {3} {4} C \left (\left (\delta^{-1} \right )^{\frac {1} {k}} - 1 \right ) \\ & \geq & C_0 \left (\delta^{-1} \right )^{\frac {1} {k}}, 
\Eea for all $0 < \delta \leq \delta_0,$ where $C_{0} = \frac {3} {4} C^{\prime} > 0.$ Now we have
\Bea \mathrm {Mdim}_{L_{\ell}^{k}} \left (C_{r}^{\ast} (\Gamma) \right ) & = & \limsup\limits_{\delta \to 0^{+}} \frac {\log D \left (\mathcal L_{1}, 2^{-1} \delta \right )} {\log 2 \delta^{-1}} \\ & \geq & \limsup\limits_{\delta \to 0^{+}} \frac {\log D \left (\lambda(\mathcal L_{1})\delta_{e}, 2^{-1} \delta \right )} {\log 2 \delta^{-1}} \\ & \geq & \limsup\limits_{\delta \to 0^{+}} \frac {\log D \left (\lambda(\delta U_{\delta})\delta_{e}, 2^{-1} \delta \right )} {\log 2 \delta^{-1}} \\ & = & \limsup\limits_{\delta \to 0^{+}} \frac {\log D \left (\lambda(U_{\delta})\delta_{e}, 2^{-1} \right )} {\log 2 \delta^{-1}} \\ 
& \geq & \limsup\limits_{\delta \to 0^{+}} \frac {\log (C_{0} \left (\delta^{-1}\right )^{\frac {1} {k}})} {\log \delta^{-1}} \\
& = & \frac {1} {k}. \Eea
For the upper bound, fix some $\frac {r} {2} < p < k.$ Then by \cite{nica2010degree}*{Proposition 2.2} it follows that $\Gamma$ has the property of rapid decay with exponent $p$ with some constant $C' > 0.$ Get hold of $n \in \mathbb N$ such that 
$$\frac {2^{p} C'} {n^{k - p}} < \delta.$$ The smallest such $n \in \mathbb N$ satisfying the inequality is given by $$n_{0} = \left \lceil \left (\frac {2^{p} C'} {\delta} \right )^{\frac {1} {k - p}} \right \rceil.$$ Then by the virtue of the inequality \eqref{eqn 2.0} we have $$\mathcal L_{1} \subseteq_{\delta} B_{n_{0}} : = \mathrm {span} \left \{\delta_{g}\ :\ \ell (g) \leq n_{0} \right \}.$$ Therefore by invoking the property of polynomial growth of degree $r \geq 1$ for the length function $\ell$, we have 
$$D \left (\mathcal L_{1}, \delta \right ) \leq \left \lvert \left \{g \in \Gamma\ : \ell (g) \leq n_{0} \right \} \right \rvert \leq \mu \left (1 + n_{0} \right )^{r},$$ for some $\mu > 0.$
Then \Bea \log D \left (\mathcal L_{1}, \delta \right ) & \leq & \log \mu + r \log \left (1 + n_{0} \right ) \\ & \leq & \log \mu + r \log \left (2 + \left (\frac {2^{p} C'} {\delta} \right )^{\frac {1} {k - p}} \right ) \\ & = & \log \mu + r \log \left (2 \delta^{\frac {1} {k - p}} + \left (2^{p} C' \right )^{\frac {1} {k - p}} \right ) + \frac {r} {k - p} \log \delta^{-1} \Eea 
Letting $\delta \to 0^{+}$ we have $$\frac {\log \mu} {\log \delta^{-1}} \to 0 \quad \mathrm {and}\ \quad \frac {\log \left (2 \delta^{\frac {1} {k - p}} + \left (2^{p} C' \right )^{\frac {1} {k - p}} \right )} {\log \delta^{-1}} \to 0.$$ 
Therefore $$\mathrm {Mdim}_{L_{\ell}^{k}} \left (C_{r}^{\ast} (\Gamma) \right ) = \limsup\limits_{\delta \to 0^{+}} \frac {\log D \left (\mathcal L_{1}, \delta \right )} {\log \delta^{-1}} \leq \frac {r} {k - p}.$$ Since the above inequality holds for all $p > \frac {r} {2},$ letting $p \to {\frac {r} {2}}^{+}$ we obtain $$\mathrm {Mdim}_{L_{\ell}^{k}} \left (C_{r}^{\ast} (\Gamma) \right ) \leq \frac {r} {k - \frac {r} {2}} = \frac {2 r} {2 k - r}.$$   
\end{proof}

Now we consider the free group on $2$-generators which has the rapid decay property with decay exponent $2$ but admits exponential growth (\cite{Haagerup}*{Lemma 1.5}). 
\begin{prop}
    For any $k > 2,$ the CQMS $(C^{\ast}_{r}(\mathbb{F}_{2}),L^{k}_{\ell})$ has infinite metric dimension.
\end{prop}
\begin{proof}
Retaining the same notations and repeating the similar computations as done in the proof of the last theorem, we find that $$D \left (U_{\delta}, 2^{-1} \right ) \geq \frac {3} {4} \left (2 \cdot 3^{\left \lfloor \left (\delta^{-1} \right )^{\frac {1} {k}} \right \rfloor} - 1 \right ) \geq \frac {3} {4} \left (\frac {2} {3} \cdot 3^{\left (\delta^{-1} \right )^{\frac {1} {k}}} - 1 \right ).$$ 
Consequently, \Bea\mathrm {Mdim}_{L_{\ell}^{k}} \left (C_{r}^{\ast} (\mathbb{F}_{2}) \right ) & \geq & \limsup\limits_{\delta \to 0^{+}} \frac {\log \left (\frac {2} {3} \cdot 3^{\left (\delta^{-1} \right )^{\frac {1} {k}}} - 1 \right )} {\log \delta^{-1}} \\ & = & \limsup\limits_{\delta \to 0^{+}} \frac {\log \left (\frac {2} {3} \cdot 3^{\left (\delta^{-1} \right )^{\frac {1} {k}}} \right )} {\log \delta^{-1}} \\ & = & \limsup\limits_{\delta \to 0^{+}} \frac {\log 3^{\left (\delta^{-1} \right )^{\frac {1} {k}}}} {\log \delta^{-1}} \\ & = & \log 3\ \limsup\limits_{\delta \to 0^{+}} \frac {\left (\delta^{-1} \right )^{\frac {1} {k}}} {\log \delta^{-1}} \\ & = & \infty. \Eea    
\end{proof}

By similar arguments one can show that the metric dimension of $F_m$ is infinite for all $m > 2.$

\vspace{5mm}

{\bf Entropy}:  We discuss entropy of a particular class of automorphisms of $C^{\ast}_{r}(\Gamma)$ for a finitely generated discrete group $\Gamma$ equipped with the word length function $\ell$ such that $\Gamma$ has the rapid decay property. Recall that any automorphism $\alpha$ of the group $\Gamma$ induces an automorphism of $C^{\ast}_{r}(\Gamma)$ (to be denoted by $\alpha$ by an abuse of notation). $\alpha$ is given on $C_{c}(\Gamma)$ by
\begin{displaymath}
    \alpha \left (\sum_{g\in F}c_{g}\delta_{g} \right )=\sum_{g\in F}c_{g}\delta_{\alpha(g)}.
\end{displaymath}
First, we discuss the case when $\Gamma=\mathbb{Z}^d$, where $\mathbb{Z}^d$ is equipped with the usual length function $\ell$ and has polynomial growth. Then this case is reduced to the case of the $C^{\ast}$-algebra $C(\mathbb{T}^d)$ by the Fourier transform which is an isomorphism between $C^{\ast}_{r}(\mathbb{Z}^d)$ and $C(\mathbb{T}^d)$. To establish this connection we note that for two isomorphic $C^{\ast}$-algebras $A,B$ any $\phi_{A}\in\textrm{Aut}(A)$ induces an automorphism $\phi_{B}$ of $B$ by composing $\phi_{A}$ with the isomorphism. As mentioned in Remark \ref{Kerrcomparison}, for any automorphism $\phi$ of $\mathbb{Z}^d$, $\mathrm{Entp}_{L^{k}_{\ell}}(\phi)=\mathrm{Entp}(\phi)$ for any $k>d$. For the next lemma, for a CQMS $(A,L_A)$ we denote the domain $\{a\in A, L(a)<+\infty\}$ of $L$ by $\mathcal{L}_{A}$. $\mathrm{Entp}_{L_{A}}$ will stand for the product entropy as defined by Kerr for a CQMS structure $(A,L_{A})$. 
\begin{lem}\label{abeliancase}
    Let $(A,L_{A})$ and $(B,L_{B})$ be two CQMS. Let us denote the domains of $L_{A}$ and $L_{B}$ by $\mathcal{L}_{A}$ and $\mathcal{L}_{B}$. Let $\alpha:A\rightarrow B$ be an isomorphism such that $\alpha(\mathcal{L}_{A})\subseteq\mathcal{L}_{B}$ and $\alpha^{-1}(\mathcal{L}_{B})\subseteq\mathcal{L}_{A}$. Then for any automorphism $\phi_{A}\in\textrm{Aut}(A)$, 
    \begin{displaymath}
        \mathrm{Entp}_{L_{A}}(\phi_{A})=\mathrm{Entp}_{L_{B}}(\phi_{B}),
    \end{displaymath}
    where $\phi_{B}=\phi_{A}\circ\alpha^{-1}$.
\end{lem}
\begin{proof}
    Straightforward.
\end{proof}
As $\mathbb{Z}^d$ has polynomial growth, we have the ${}_c$Lip-norm $L^{k}_{\ell}$ for large enough $k$ whose domain is $C_{c}(\mathbb{Z}^d)$. On the other hand the $C^{\ast}$-algebra $C(\mathbb{T}^d)$ can be equipped with the ${}_c$Lip-norm coming from an ergodic action of $\mathbb{T}^{d}$ exactly like $C(\mathbb{T}^{d}_{\theta})$ as done in \cite{Kerr}. The following can be obtained by adapting the arguments done for $C(\mathbb{T}^{d}_{\theta})$ in \cite{Kerr}:
\begin{thm}\label{commtorus}
     Let $\ell$ be the length function on $\mathbb T^{d}$ obtained by taking the distance to $0$ with respect to the quotient norm on $\mathbb T^{d}$ induced from the Euclidean norm on $\mathbb R^{d}$ scaled by $2 \pi.$ For generating unitaries $u_1, \cdots, u_d$ of $C(\mathbb T^{d}),$ let $\gamma$ be the ergodic action of $\mathbb T^{d}$ on $C(\mathbb T^{d})$ determined by 
     $$\gamma_{\left (t_1, \cdots, t_d \right )} \left (u_j \right ) = e^{2 \pi i t_j} u_j.$$ Let $L$ be the ${}_c\mathrm{Lip}$-norm on $C(\mathbb T^{d})$ arising from the length function $\ell$ and the ergodic action $\gamma;$ $\alpha_{T}$ be an automorphism on $C \left (\mathbb T^{d} \right )$ induced by $T \in GL_{d} (\mathbb Z).$ Then $$\mathrm {Entp}_{L} \left (\alpha_{T} \right ) = \sum\limits_{\left \lvert \lambda_i \right \rvert > 1} \log \left \lvert \lambda_i \right \rvert,$$ where $\lambda_i$'s are eigenvalues of $T$ counted with spectral multiplicity.  
\end{thm}
The Fourier transform 
$\mathcal{F}:C^{\ast}_{r}(\mathbb{Z}^d)\rightarrow C(\mathbb{T}^d)$ maps $C_{c}(\mathbb{Z}^{d})$ to the polynomials in the generating unitaries. The automorphism group of the group $\mathbb{Z}^{d}$ is $GL_{d}(\mathbb{Z})$. So any element $\phi\in GL_{d}(\mathbb{Z})$ gives an automorphism of $C^{\ast}_{r}(\mathbb{Z}^d)$. Such an automorphism is transported via the Fourier transform to the toral automorphisms of $C(\mathbb{T}^d)$. Therefore, Lemma \ref{abeliancase} and Theorem \ref{commtorus} combine to give us the following:
\begin{thm} \label{entpZ}
 Let $\alpha_{T}$ be an automorphism of $C \left (\mathbb T^{d} \right )$ induced by $T \in GL_{d} \left (\mathbb Z \right ).$ Let $\beta_{T}$ be the automorphism on $C_{r}^{\ast} \left (\mathbb Z^{d} \right )$ associated to $\alpha_{T}.$ Then for all $k > d$ $$\mathrm {Entp}_{L_{\ell}^{k}} \left (\beta_{T} \right ) = \sum\limits_{\left \lvert \lambda_i \right \rvert > 1} \log \left \lvert \lambda_i \right \rvert,$$ where $\lambda_i$'s are eigenvalues of $T$ counted with spectral multiplicity. Consequently, by Remark \ref{Kerrcomparison}, \begin{displaymath}\mathrm{Entp}\left (\beta_{T} \right ) = \sum\limits_{\left \lvert \lambda_i \right \rvert > 1} \log \left \lvert \lambda_i \right \rvert.\end{displaymath}
\end{thm}

\begin{proof}
Let $L$ be the ${}_c\mathrm{Lip}$-norm on $C (\mathbb T^{d})$ as in Theorem \ref{commtorus} and let $\mu : = \sum\limits_{\left \lvert \lambda_i \right \rvert > 1} \log \left \lvert \lambda_i \right \rvert,$ where $\lambda_i$'s are eigenvalues of $T$ counted with spectral multiplicity. Let $\widehat {L}$ be the restriction of $L$ to the set of all polynomials $\mathcal {P} (\mathbb T^{d})$ in the generating unitaries of $\mathbb T^{d}$. Then by the virtue of Lemma \ref{abeliancase}, we have \begin{equation} \label{ent_1} \mathrm {Entp}_{L_{\ell}^{k}} \left (\beta_{T} \right ) = \mathrm {Entp}_{\widehat{L}} \left (\alpha_{T} \right ), \end{equation} for all $k > d.$ Recall that the lower bound of Kerr's entropy of $\alpha_{T}$ is obtained by exploiting only the finite set of polynomials in the generating unitaries of $\mathbb T^{d},$ namely by considering the sets of the form $$U_K : = \left \{u_1^{k_1} u_2^{k_2} \cdots u_d^{k_d}\ :\ \left (k_1, k_2, \cdots, k_d \right ) \in K \right \},$$ for any finite subset $K \subseteq \mathbb Z^{d}$. Therefore, following the proof of \cite{Kerr}*{Proposition 7.2}, we get \begin{equation} \label{ent_2} \mathrm {Entp}_{\widehat{L}} \left (\alpha_{T} \right ) \geq \mu. \end{equation} Also it is clear that $\mathrm {Entp}_{\widehat{L}} \left (\alpha_{T} \right ) \leq \mathrm {Entp}_{L} \left (\alpha_{T} \right ).$ So by invoking Theorem \ref{commtorus}, it follows that \begin{equation} \label{ent_3} \mathrm {Entp}_{\widehat{L}} \left (\alpha_{T} \right ) \leq \mu. \end{equation} Thus by combining \eqref{ent_1}, \eqref{ent_2} and \eqref{ent_3}, the desired result follows. 
\end{proof}
So, we are left with the cases of non-abelian groups. We start by obtaining a lower bound of the product entropy of an automorphism. To get the lower bound, we do not have to use any CQMS structure on the group $C^{\ast}$-algebra. To that end, we recall the definition of algebraic entropy of an automorphism of a discrete group. Let $\Gamma$ be a discrete group with an automorphism $\phi$. For any finite subset $\mathcal{F}$ of $\Gamma$, denote the set $\{\mathcal{F}\phi(\mathcal{F})...\phi^{n-1}(\mathcal{F})\}$ by $\mathcal{F}_n$, where as usual for two sets $E_1,E_2\subset\Gamma$, $E_1E_2=\{g_{1}g_{2}:g_{i}\in E_{i} \ \mathrm{for} \ i=1,2\}$. Denoting the cardinality of any set $E\subseteq\Gamma$ by $\lvert E\rvert$, we recall the following definition:
\begin{defn}\label{algentp}(\cite{MR0385834}*{Definition 1.1})
    The algebraic entropy of an automorphism $\phi$ of a discrete group $\Gamma$ is defined to be
    \begin{align*}
        && h_{\mathrm{alg}}(\phi):=\sup_{\mathcal{F}\subseteq\Gamma \ \textrm{finite}}\limsup\limits_{n\to\infty}\frac{1}{n}\log\lvert\mathcal{F}_n\rvert.
    \end{align*}
\end{defn}

\begin{rem}
Note that the above limit superior in the definition of the algebraic entropy is actually a limit. This is because the sequence $\{c_n \}_{n \geq 1}$ defined by $c_n : = \log \left \lvert \mathcal F_n \right \rvert$ has the subadditivity property i.e. $c_{n + m} \leq c_n + c_m$ for all $m, n \in \mathbb N$ \cite{Dikranjan2013TopologicalEA}*{Lemma 5.1.1.}. So by Fekete's lemma, the sequence $\left \{\frac {c_n} {n} \right \}_{n \geq 1}$ has limit and in fact $\lim\limits_{n \to \infty} \frac {c_n} {n} = \inf\limits_{n \in \mathbb N} \frac {c_n} {n}.$
\end{rem}
\begin{thm}\label{lowerbound}
Let $\Gamma$ be a finitely generated discrete group with an automorphism $\phi$. Then
\begin{align*}
&& \mathrm{Entp}(\phi)\geq h_{\mathrm{alg}}(\phi).
\end{align*}
\end{thm}
\begin{proof}
Let us choose $\varepsilon > 0.$ By the definition of $h_{\mathrm{alg}}(\phi)$, we can get hold of some finite subset $\mathcal{F}^{\varepsilon} \subseteq \Gamma$ and a subsequence $\left \{n_k \right \}_{k \geq 1}$ of natural numbers such that $$\frac {1} {n_k} \log \left \lvert \mathcal{F}^{\varepsilon}_{n_k} \right \rvert \geq h_{\mathrm{alg}}(\phi) - \varepsilon.$$
In other words, for all $k \geq 1$
$$\left \lvert \mathcal{F}^{\varepsilon}_{n_k} \right \rvert \geq e^{n_k (h_{\mathrm{alg}}(\phi) - \varepsilon)}.$$
Let $\Omega_{\varepsilon} \subseteq C_c (\Gamma)_1$ be the finite subset given by $$\Omega_{\varepsilon} : = \left \{\delta_{g}\ :\ g \in \mathcal{F}^{\varepsilon} \right \},$$ and, for $n \geq 1,$ define $$\Omega_{\varepsilon}^{(n)} : = \prod\limits_{i = 1}^{n} \varphi^{i - 1} \left (\Omega_{\varepsilon} \right ).$$ Then clearly $$\Omega_{\varepsilon}^{(n)} = \left \{\delta_{g}\ :\ g \in \mathcal{F}^{\varepsilon}_{n} \right \},$$ for all $n \geq 1.$ Invoking Theorem \ref{Voiculescu} it follows that  $$D \left (\lambda \left (\Omega_{\varepsilon}^{(n_k)} \right ) \delta_e, \delta \right ) \geq \left (1 - \delta^{2} \right ) e^{n_k (h_{\mathrm{alg}}(\phi) - \varepsilon)},$$ for all $k \geq 1$ and for all $\delta > 0.$ Thus for all $\delta > 0$ we have 
\Bea
\mathrm {Entp} \left (\psi, \Omega_{\varepsilon}, \delta \right ) & \geq & \limsup\limits_{k \to \infty} \frac {1} {n_k} \log D \left (\Omega_{\varepsilon}^{(n_k)}, \delta \right ) \\ & \geq & \limsup\limits_{k \to \infty} \frac {1} {n_k} \log D \left (\lambda \left (\Omega_{\varepsilon}^{(n_k)} \right ) \delta_e, \delta \right ) \\ & \geq & \lim\limits_{k \to \infty} \frac {1} {n_k} \log \left (\left (1 - \delta^{2} \right ) e^{n_k (h_{\mathrm{alg}}(\phi) - \varepsilon)} \right ) \\ & = & h_{\mathrm{alg}}(\phi) - \varepsilon.
\Eea
Taking supremum over all finite subset $\Omega \subseteq C_c (\Gamma)_1$ and $\delta > 0$ it follows that $$\mathrm {Entp} (\phi) \geq h_{\mathrm{alg}}(\phi) - \varepsilon.$$ Finally since $\varepsilon > 0$ was arbitrary, letting $\varepsilon \to 0^{+},$ the desired result follows.
\end{proof}

\begin{rem}
Any $T\in GL_{d}(\mathbb{Z})$ is an automorphism of $\mathbb{Z}^d$. By \cite{MR540634}*{Example 2}, $h_{\mathrm{alg}}(T)=\sum\limits_{\left \lvert \lambda_i \right \rvert > 1} \log \left \lvert \lambda_i \right \rvert,$ where $\lambda_i$'s are eigenvalues of $T$ counted with spectral multiplicity. Note that $T$ produces the automorphism $\beta_{T}$ of $C^{\ast}_{r}(\mathbb{Z}^d)$ in the notation of Theorem \ref{entpZ}. Thus the lower bound obtained for the product entropy in Theorem \ref{lowerbound} is typically strict by Theorem \ref{entpZ}.
\end{rem}
Now we shall get an upper bound of the product entropy of an automorphism by exploiting the CQMS structures on the group $C^{\ast}$-algebras coming from the ${}_c\mathrm{Lip}$-norms $L^{k}_{\ell}$. Let us recall the notion of a `geometric' entropy of an automorphism of a finitely generated discrete group $\Gamma$ equipped with some word length $\ell$ in the sense of \cite{MR4168745}. Note that in \cite{MR4168745}, the geometric entropy is referred to as algebraic entropy. 
\begin{defn}\label{geoment} (see \cite{MR4168745}) Given any automorphism $\phi$ of $\Gamma$, the geometric entropy $h_{\mathrm{geo}}(\phi)$ is defined to be the following number:
\begin{align*}
    && h_{\mathrm{geo}}(\phi):=\log\Big(\sup_{g\in\Gamma}\{\limsup\limits_{n}\ell(\phi^{n}(g))^{\frac{1}{n}}\}\Big).
\end{align*}
\end{defn}
We shall first prove two auxilliary lemmas. To that end, for any $\phi\in\mathrm{Aut}(\Gamma)$, we shall denote the quantity $\sup_{g\in\Gamma}\{\limsup\limits_{n}\ell(\phi^{n}(g))^{\frac{1}{n}}\}$ by $\mathrm{Gr}(\phi)$. Note that since $\ell$ is a word length, $\mathrm{Gr}(\phi)\geq 1$ and consequently $h_{\mathrm{geo}}(\phi)=\log(\mathrm{Gr}(\phi))\geq 0$.

\begin{lem} \label{lem 2.23}
Let $\Gamma$ be a finitely generated discrete group with the rapid decay property with decay exponent $r>0$; $\ell$ be a word length function so that for any natural number $k>r$, $(C^{\ast}_{r}(\Gamma),L^{k}_{\ell})$ is a CQMS; $\phi$ be an automorphism of $\Gamma$; $\Omega$ be a finite subset of $C_{c}(\Gamma)$ and $M_{\Omega}=\max\limits_{f\in \Omega}\{\lvert\mathrm{supp}(f)\rvert\}.$ Then given any $\varepsilon > 0$ there exists $q \in \mathbb N$ such that for all $n \geq q$ and for all $f \in \Omega$ 
\begin{displaymath}
    L^{k}_{\ell}(\phi^{n}(f))\leq \sqrt{\M_{\Omega}}\ (\mathrm{Gr}(\phi) + \varepsilon)^{kn}L^{k}_{\ell}(f), \ \forall \ f\in\Omega.
\end{displaymath}
\end{lem}
\begin{proof}
  For all $g \in \bigcup\limits_{f \in \Omega} \mathrm {supp} (f)$, by the definition of limit superior, we can get hold of some $q \in \mathbb N$ such that for all $n \geq q$ and for all $g \in \bigcup\limits_{f \in \Omega} \mathrm {supp} (f)$
$$\ell (\phi^{n} (g))^{\frac {1} {n}} < \mathrm{Gr}(\phi) + \varepsilon.$$ In other words, for all $g \in \bigcup\limits_{f \in \Omega} \mathrm {supp} (f)$ and for all $n \geq q$ $$\ell (\phi^{n} (g)) < (\mathrm{Gr}(\phi) + \varepsilon)^{n} \leq (\mathrm{Gr}(\phi) + \varepsilon)^{n} \ell (g),$$ since $\ell$ is a word length function. Now for any $f = \sum\limits_{g \in \Gamma} c_{g} \delta_{g} \in \Omega$ and for any $n \geq q$, we have  \Bea
        L^{k}_{\ell} \left (\phi^{n}(f) \right )& = & L^{k}_{\ell} \left (\sum c_{g}\delta_{\phi^{n}(g)} \right )\\
        & \leq & \sum \left \lvert c_{g} \right \rvert L^{k}_{\ell} \left (\delta_{\phi^{n}(g)} \right )\\ 
        & \leq & \sum \lvert c_{g}\rvert (\ell(\phi^{n}(g)))^{k}\ (\mathrm{by\ Lemma}\ \ref{deltag}) \\
        & \leq & (\mathrm{Gr}(\phi) + \varepsilon)^{nk}\sum |c_{g}|(\ell(g))^{k}\\
        & \leq & (\mathrm{Gr}(\phi) + \varepsilon)^{nk} (\sum \lvert c_{g}\rvert^{2}(\ell(g))^{2k})^{\frac{1}{2}}\sqrt{M_{\Omega}}\\
        & \leq & \sqrt{\M_{\Omega}}\ (\mathrm{Gr}(\phi) + \varepsilon)^{kn}L^{k}_{\ell}(f) \ (\mathrm{by} \ \mathrm{inequality} \ \ref{Lbound}).
    \Eea
\end{proof}
\begin{lem} \label{lem 2.24}
Let $\Gamma, \phi,\ell$ be as in the previous lemma. For any finite set $\Omega\subseteq C_{c}(\Gamma)_1$, let $M_{\Omega}=\max\limits_{f\in \Omega}\{\lvert\mathrm{supp}(f)\rvert\}$ and $\Omega_{n}=\prod_{i=0}^{n}\phi^{i}(\Omega)$. Then given any $\varepsilon > 0$ there exists $q \in \mathbb N$ such that for all $n \geq q$ and for all $a \in \Omega_n$,
\begin{displaymath}
    L_{\ell}^{k} (a)\leq C \left (\sqrt{\M_{\Omega}} \right )^{k}( \mathrm{Gr}(\phi)+\varepsilon)^{k (n - 1)} (n - q + 1)^k,
\end{displaymath}
for some constant $C > 0$ depending upon $\Omega$.
\end{lem}

\begin{proof}
Let us choose $\varepsilon > 0$ arbitrarily. Then from Lemma \ref{lem 2.23}, we can get hold of $q \in \mathbb N$ such that for all $n \geq q$ 
\begin{displaymath}
    L^{k}_{\ell}(\phi^{n}(f))\leq \sqrt{\M_{\Omega}}\ (\mathrm{Gr}(\phi) + \varepsilon)^{kn}L^{k}_{\ell}(f), \ \forall \ f\in\Omega.
\end{displaymath}
Choose any $n \geq q$ and let $a \in \Omega_n.$ Then there exist $f_1, f_2, \cdots, f_n \in \Omega$ such that $a = f_1 \phi(f_2) \cdots \phi^{n - 1} (f_n).$ Set $h = f_1 \phi \left (f_2 \right ) \cdots \phi^{q - 1} \left (f_q \right ).$ Let $$C' : = \max \left \{L_{\ell}^{s} \left (f) \right )\ :\ f \in \Omega,\ 1 \leq s \leq k \right \},$$ and, $$C_1' : = \max \left \{L_{\ell}^{s} \left (\prod\limits_{i = 1}^{q} \phi^{i - 1} \left (f_i \right ) \right )\ :\ f_i \in \Omega,\ 1 \leq i \leq q,\ 1 \leq s \leq k \right \}$$ and set $C_1 : = \max \left \{1, C_1' \right \}$ and $C_2 : = \max \left \{1, {C'}^k \right \}.$ 
Then by Corollary \ref{LeibnitzL} we have 
\Bea
L_{\ell}^{k} (a) & = & L_{\ell}^{k} \left (h \prod\limits_{i = 0}^{n - q - 1} \phi^{q + i} (f_{q + i + 1}) \right ) \\ & \leq & \sum\limits_{r_1 + r_2 + \cdots + r_{n - q + 1} = k} \binom {k} {r_1\ r_2\ \cdots\ r_{n - q + 1}} L_{\ell}^{r_1} (h) \prod\limits_{i = 0}^{n - q - 1} L_{\ell}^{r_{i + 2}} \left (\phi^{q + i} (f_{q + i + 1}) \right ) \\ & \leq & C_1 \left (\sqrt {M_{\Omega}} \right )^{k} \sum\limits_{r_1 + r_2 + \cdots + r_{n - q + 1} = k} \binom {k} {r_1\ r_2\ \cdots\ r_{n - q + 1}} \prod\limits_{i = 0}^{n - q - 1} (\mathrm{Gr}(\phi) + \varepsilon)^{r_{i + 2} (q + i)} L_{\ell}^{r_{i + 2}} \left ( f_{q + i + 1} \right ) \\ & \leq & C_1 C_2 \left (\sqrt {M_{\Omega}} \right )^{k} (\mathrm{Gr}(\phi) + \varepsilon)^{k (n - 1)} \sum\limits_{r_1 + r_2 + \cdots + r_{n - q + 1} = k} \binom {k} {r_1\ r_2\ \cdots\ r_{n - q + 1}} \\ & = & C \left (\sqrt {M_{\Omega}} \right )^{k} (\mathrm{Gr}(\phi) + \varepsilon)^{k (n - 1)} (n - q + 1)^{k},
\Eea
where $C : = C_1 C_2 > 0.$
\end{proof}

\begin{thm}\label{upperbound}
    Let $\Gamma$ be a finitely generated discrete group of rapid decay with decay exponent say $r$ equipped with a word length function $\ell$. For any $k>r$, if $\mathrm{Mdim}_{L^{k}_{\ell}}(C^{\ast}_{r}(\Gamma))=d<+\infty$, then for any automorphism $\phi$ of $\Gamma$, 
    \begin{align*}
        \mathrm{Entp}(\phi)\leq kd\ h_{\mathrm{geo}}(\phi).
    \end{align*}
\end{thm}

\begin{proof}
We will exploit Lemma \ref{lem 2.24} to prove this theorem. Fix some $k > r$ and some finite subset $\Omega \subseteq {C_c (\Gamma)}_1$ and let $\varepsilon, \delta > 0.$ Let $\mathcal L_{1}$ be the unit ball of $C_c (\Gamma)$ relative to $L_{\ell}^{k},$ where $\ell$ is some word length function on $\Gamma$ as in Lemma \ref{lem 2.24}. Following the same notation as in Lemma \ref{lem 2.24}, let $\widetilde {C} : = C \left (\sqrt {M_{\Omega}} \right)^{k}$ and let $\lambda_{\varepsilon} : = \mathrm{Gr}(\phi) + \varepsilon.$ Note that $\lambda_{\varepsilon} \geq 1 + \varepsilon > 1.$ Then it follows from Lemma \ref{lem 2.24} that $$\Omega_n : = \prod\limits_{i = 1}^{n} \phi^{i - 1} (\Omega) \subseteq \widetilde {C} \lambda_{\varepsilon}^{k (n - 1)} (n - q + 1)^k \mathcal L_1 \subseteq \widetilde {C} \lambda_{\varepsilon}^{n k} n^k \mathcal L_1,\ \forall n \geq q,$$ for some $q \in \mathbb N$ depending upon $\varepsilon > 0.$ Then for all $n \geq q$ we have
\Bea
D \left (\Omega_n, \delta \right ) & = & D \left (\frac {\Omega_n} {\widetilde{C} \lambda_{\varepsilon}^{nk} n^k}, \frac {\delta} {\widetilde{C} \lambda_{\varepsilon}^{nk} n^{k}} \right ) \\ & \leq & D \left (\mathcal L_1, \frac {\delta} {\widetilde{C} \lambda_{\varepsilon}^{nk} n^{k}} \right ).
\Eea
Thus
\Bea
\mathrm {Entp} (\phi, \Omega, \delta) & = & \limsup\limits_{n \to \infty} \frac {1} {n} \log D \left (\Omega_n, \delta \right ) \\ & \leq & \limsup\limits_{n \to \infty} \frac {1} {n} \log D \left (\mathcal L_1, \frac {\delta} {\widetilde{C} \lambda_{\varepsilon}^{nk} n^{k}} \right ) \\ & = & \limsup\limits_{n \to \infty} \frac {\log D \left (\mathcal L_1, \frac {\delta} {\widetilde{C} \lambda_{\varepsilon}^{n k} n^k} \right )} {\log \left (\delta^{-1} \widetilde{C} \lambda_{\varepsilon}^{n k} n^k \right )} \lim\limits_{n \to \infty} \frac {\log \left (\delta^{-1} \widetilde{C} \lambda_{\varepsilon}^{n k} n^k \right )} {n} \\ & = & d \lim\limits_{n \to \infty} \frac {\log \lambda_{\varepsilon}^{n k}} {n} \\ & = & k d \log \lambda_{\varepsilon} \\ & = & k d \log (\mathrm{Gr}(\phi) + \varepsilon).
\Eea
This is true for any $\varepsilon > 0.$ So letting $\varepsilon \to 0^{+},$ the desired result follows.
\end{proof}
As a corollary, we show that the product entropy of inner automorphisms of a finitely generated group are zero, provided the metric dimension is finite. Recall that for $z\in\Gamma$, the inner automorphism $\mathrm{Ad}_{z}$ is given by $g\mapsto zgz^{-1}$. Then it is easy to check that $(\mathrm{Ad}_{z})^n=\mathrm{Ad}_{z^n}$ for all $n$.
\begin{cor}\label{innerzero}
 Let $\Gamma$ be a finitely generated discrete group of rapid decay with decay exponent say $r$ equipped with a word length function $\ell$. For any $k>r$, if $\mathrm{Mdim}_{L^{k}_{\ell}}(C^{\ast}_{r}(\Gamma))=d<+\infty$, $\mathrm{Entp}(\mathrm{Ad}_z)=0$ for any $z\in\Gamma$.   
\end{cor}
\begin{proof}
    By Theorem \ref{upperbound}, it suffices to show that $h_{\mathrm{geo}}(\mathrm{Ad}_{z})=0$ for all $z$. To that end, for any fixed $g\in\Gamma$, by subadditivity and symmetry of the length function, $\ell\Big((\mathrm{Ad}_{z})^{n}(g)\Big)\leq (\ell(g)+\ell(z))2n$ for all $n$. Therefore, $\mathrm{Gr}(\mathrm{Ad}_{z})\leq\sup\limits_{g\in\Gamma} \limsup\limits_{n}\Big((\ell(g)+\ell(z))2n\Big)^{\frac{1}{n}} \leq 1$. Taking logarithm, we get $h_{\mathrm{geo}}(\mathrm{Ad}_{z})=0$.
\end{proof} 
Now we consider the free group on $m$-generators $\mathbb{F}_{m}$ to show that the assumption of finiteness of metric dimension in Corollary \ref{innerzero} is crucial.
 \begin{lem}\label{freegroupinner}
For $m \geq 2,$ let $\mathbb{F}_{m}$ denote the free group on $m$-generators. Then $$\mathrm {Entp} \left (\mathrm {id} \right ) = \infty,$$ where $\mathrm {id}$ denotes the identity automorphism of $C_r^{\ast} \left (\mathbb{F}_{m} \right ).$
\end{lem}
\begin{proof}
Follows from Theorem \ref{lowerbound} and the well known fact \cite{Dikranjan2013TopologicalEA}*{Proposition 5.3.13.} that the algebraic entropy of the trivial automorphism of $\mathbb{F}_{m}$ is $+\infty$.\end{proof}
In fact, \cite{Dikranjan2013TopologicalEA}*{Proposition 5.3.13.}  says that the algebraic entropy of the identity automorphism of a finitely generated discrete group with exponential growth is always $+\infty$. We shall use this to show that the metric dimension of the $C^{\ast}$-algebras of groups with exponential growth is generically $+\infty$. Recall that the pair $(C^{\ast}_{r}(\Gamma),L^{k}_{\ell})$ is a CQMS for large $k$'s if $\Gamma$ has the rapid decay property. There are many examples of group with exponential growth admitting rapid decay property. Examples include the free groups \cite{Haagerup}*{Lemma 1.5}, Gromov's hyperbolic groups \cites{MR943303, de1988groupes}.
\begin{thm}\label{expoinfty}
    Let $\Gamma$ be a finitely generated discrete group with exponential growth such that $\Gamma$ has the rapid decay property with some decay exponent $r$. Then $\mathrm{Mdim}_{L^{k}_{\ell}}(C^{\ast}_{r}(\Gamma))=+\infty$ for all $k>r$.  
\end{thm}
\begin{proof}
    We denote the identity automorphism on $\Gamma$ by $\mathrm{id}_{\Gamma}$. By  \cite{Dikranjan2013TopologicalEA}*{Proposition 5.3.13.}, $h_{\mathrm{alg}}(\mathrm{id_{\Gamma}})=+\infty$. So by Theorem \ref{lowerbound}, $\mathrm{Entp}(\mathrm{id}_{\Gamma})=+\infty$. Then  $\mathrm{Mdim}_{L^{k}_{\ell}}(C^{\ast}_{r}(\Gamma))=+\infty$ for all $k>r$ by Corollary \ref{innerzero}. 
\end{proof}
\subsection{Product entropy of groups with the polynomial growth property}
Now we return to the regime of group $C^{\ast}$-algebras with finite metric dimension. By Theorem \ref{finiteMdim}, $(C^{\ast}_{r}(\Gamma),L^{k}_{\ell})$ has finite metric dimension for suitable $k$'s if $\Gamma$ is a finitely generated discrete group with polynomial growth. 
 \begin{thm}\label{boundpolygrowth}
   Let $\Gamma$ be a finitely generated discrete group having the property of polynomial growth with growth exponent $r$; $\phi$ be some automorphism of $\Gamma$. Then 
   \begin{align*}
       \mathrm{Entp}(\phi)\leq rh_{\mathrm{geo}}(\phi).
   \end{align*}
 \end{thm} 
\begin{proof}
    For any natural number $k>r$, $(C^{\ast}_{r}(\Gamma),L^{k}_{\ell})$ is a CQMS for some word length function $\ell$. By Theorem \ref{finiteMdim}, $\mathrm{Mdim}_{L^{k}_{\ell}}(C^{\ast}_{r}(\Gamma))\leq \frac{2r}{2k-r}$. Then by Theorem \ref{upperbound}, for all $k>r$,
    \begin{align*}
        \mathrm{Entp}(\phi)\leq \frac{2kr}{2k-r}h_{\mathrm{geo}}(\phi).
    \end{align*}
    Therefore, letting $k\rightarrow\infty$, we get the desired upper bound.
\end{proof}
Therefore, combining Theorem \ref{lowerbound} and Theorem \ref{boundpolygrowth}, we get the main theorem of the paper$:$
\begin{maintheorem}\label{mainthm}
 Let $\Gamma$ be a finitely generated discrete group having the property of polynomial growth with growth exponent $r$; $\phi$ be some automorphism of $\Gamma$. Then 
   \begin{align*}
       h_{\mathrm {alg}} (\phi) \leq \mathrm{Entp}(\phi)\leq rh_{\mathrm{geo}}(\phi).
   \end{align*}
\end{maintheorem}
\begin{rem}\label{Lips}
    (1) Note that as the function $k\mapsto\frac{2rk}{2k-r}$ is a decreasing function of $k$ for fixed $r$, we have exerted the full force of $k$ in the ${}_c$Lip-norm $L^{k}_{\ell}$ to obtain the upper bound in the Theorem \ref{boundpolygrowth} by letting $k\rightarrow\infty$.\\
    (2) In a previous version of this article, we had obtained an upper bound in terms of the Lipschitz constant of an automorphism with respect to the length function. More precisely, if $\Gamma$ is a finitely generated discrete group having the polynomial growth property with growth exponent $r$ with a length function $\ell$ and $\phi$ is an automorphism of $\Gamma$ such that $\ell(\phi(g))\leq \lambda \ell(g)$, for some constant $\lambda$, then $\mathrm{Entp}(\phi)\leq r\log\lambda$. But the upper bound obtained in this paper is an improvement and in fact, independent of the choice of word length functions in the same quasi-isometry class. 
 \end{rem}
     Now we discuss examples. Let $\Gamma$ be a finitely generated discrete group with the polynomial growth property. Then by \cite{Gromov}, $\Gamma$ is a virtually nilpotent group i.e. $\Gamma$ is a finite extension of a nilpotent subgroup $N\subset\Gamma$. Let $\phi$ be an automorphism of $\Gamma$ such that $\phi$ is not eventually trivial and keeps $N$ invariant. Then by \cite{MR4168745}*{Theorem 2.8}, $h_{\mathrm{geo}}(\phi)=h_{\mathrm{geo}}(\phi\vert_{N})$. Therefore, the geometric entropy of an automorphism essentially boils down to the geometric entropy of nilpotent groups. So, we assume $\Gamma$ is a torsionless nilpotent finitely generated discrete group. Recall the lower central series 
     \begin{displaymath}
         \Gamma_{1}\supseteq\Gamma_{2}\supseteq\ldots\supseteq\Gamma_{n}=\{e\},
     \end{displaymath}
    for some finite $n$ where $\Gamma_{i}$ is defined recursively by
    \begin{displaymath}
          \Gamma_{1}:=\Gamma, \ \Gamma_{2}:=[\Gamma,\Gamma], \ \Gamma_{3}:=[\Gamma_{2},\Gamma],\ldots, \ \Gamma_{i+1}:=[\Gamma_{i},\Gamma]. 
    \end{displaymath}
    Any automorphism $\phi$ of $\Gamma$ canonically induces automorphism $\phi_{i}:\Gamma_{i}/\Gamma_{i+1}\rightarrow\Gamma_{i}/\Gamma_{i+1}$ for all $i$. Each $\phi_{i}$ is essentially an automorphism of some $\mathbb{Z}^{d_{i}}$ and therefore an element of $GL_{d_{i}}(\mathbb{Z})$ for some natural number $d_{i}$ for each $i$. We denote the matrix of $\phi_{i}$ by $\phi_{i}$ itself with some abuse of notation. If we denote the spectral radius of $\phi_{i}$ by $\rho_{i}$, then by \cite{MR4168745}*{Theorem 3.7},
    \begin{equation} \label{abelianization}
        h_{\mathrm{geo}}(\phi)=\log\rho_1.
    \end{equation}
   As for the lower bound, for any automorphism $\phi$ of a finitely generated discrete group $\Gamma$ that sends a subgroup $H$ onto itself, we denote the induced automorphism on $\Gamma/H$ by $\hat {\phi}$. Then by \cite{MR0385834}*{Proposition 1.5}, 
\begin{equation} \label{algentlow} h_{\mathrm{alg}}(\phi)\geq\mathrm{max}\{h_{\mathrm{alg}}(\phi\vert_{H}),h_{\mathrm{alg}}(\hat {\phi})\} \end{equation}.
\hspace{-4.5mm} Therefore, for a nilpotent group $\Gamma$, \begin{displaymath}h_{\mathrm{alg}}(\phi)\geq\mathrm{max}_{i}\{h_{\mathrm{alg}}(\phi_{i})\}.\end{displaymath} Each $\phi_{i}$ is an automorphism of some $\mathbb{Z}^{d_{i}}$ and therefore, by \cite{MR540634}*{Theorem 6}, their algebraic entropies are known in principle. So, for a torsionless virtually nilpotent group $\Gamma$ and an automorphism $\phi$ which is not eventually trivial, we have more or less complete information about the lower and upper bound of $\mathrm{Entp}(\phi)$. As an application we shall produce a class of outer automorphisms on a finitely generated discrete group $\Gamma$ such that the product entropy is zero. \\ 
 \indent  Let $\phi\in GL_{d}(\mathbb{Z})$. Then $\mathbb Z^{d} \rtimes_{\varphi} \mathbb Z$ is the finitely generated discrete group generated by $\mathbb{Z}^{d}$ and $\mathbb{Z}$ such that $txt^{-1}=\phi(x)$ for all $x\in\mathbb{Z}^{d},$ where $t$ denotes the generator of $\mathbb{Z}$. It is easy to see that any $\psi\in GL_{d}(\mathbb{Z})$ is a natural automorphism of $\mathbb{Z}^{d}\rtimes_{\phi}\mathbb{Z}$ if $\phi\psi=\psi\phi$. The automorphism is given by
\begin{displaymath}
    x\mapsto\psi(x), \ x\in\mathbb{Z}^d; \ t\mapsto t.
\end{displaymath}
It is clear that the group is generated by the generators of $\mathbb{Z}^d$ (say $x_{1},x_{2},\ldots,x_{d}$) and $t$ and has the natural word length function say $\ell$. \\Consider the non-abelian group $\Gamma=\mathbb{Z}^d\rtimes_{-I}\mathbb{Z}$. If we denote the generators of $\mathbb{Z}^d$ by $x_{1},x_{2},\ldots,x_{d}$ and the generator of $\mathbb{Z}$ by $t$, then it has an abelian subgroup $H=\langle x_{1},x_2,\ldots,x_{d},t^2\rangle\cong\mathbb{Z}^{d+1}$. In fact, $[\Gamma:H]=2$. Consider an upper triangular matrix $\phi\in GL_{d}(\mathbb{Z})$ with diagonal entries $1$ and atleast one non-diagonal entry non-zero. Then $\phi$ gives an outer automorphism $\widetilde {\phi}$ of $\Gamma$ and it also restricts to $H$ to produce an automorphism of $\mathbb{Z}^{d+1}$. Moreover, it is easy to see that $\Gamma$ is torsion-free and $\widetilde{\phi}$ is not eventually trivial. Therefore, by \cite{MR4168745}*{Theorem 2.8}, $h_{\mathrm{geo}}(\widetilde {\phi})=h_{\mathrm{geo}}(\widetilde {\phi} \vert_{\mathbb{Z}^{d+1}})$. But the matrix of $\widetilde {\phi} \vert_{\mathbb{Z}^{d+1}}$ (with respect to the generators of $H$) is given by the block matrix $\begin{pmatrix} \phi & 0 \\ 0 & 1 \end{pmatrix}$ which is again an upper triangular matrix with diagonal entries $1$. Therefore, the spectral radius of $\widetilde{\phi} \rvert_{\mathbb Z^{d + 1}}$ is $1$. It is a standard fact that $\mathbb{Z}^d\rtimes_{-I}\mathbb{Z}$ has polynomial growth. Hence by Theorem \ref{boundpolygrowth}, we have the following$:$
 \begin{thm}
  Let $\phi$ be an upper triangular matrix in $GL_{d}(\mathbb{Z})$ with diagonal entries $1$ and at least one non-diagonal entry non-zero. Then $\widetilde {\phi}$ is an outer automorphism of $\mathbb{Z}^d\rtimes_{-I}\mathbb{Z}$ and $\mathrm{Entp}(\widetilde{\phi})=h_{\mathrm{geo}}(\widetilde {\phi})=0$.    
 \end{thm}
Let us produce a concrete example of an automorphism $\phi$ of a group $\Gamma$ such that $\mathrm{Entp}(\phi)$ is non-zero and finite. To that end, consider the discrete Heisenberg group $H_3 (\mathbb Z).$ It is described as the set $\mathbb Z^{3}$ of integer triples endowed with the following multiplication$:$
$$(x, y, z) \cdot (u, v, w) = (x + u + y w, y + v, z + w).$$ Equivalently, $H_3 (\mathbb Z)$ may be presented as follows$:$
$$\left \langle \alpha, \beta\ :\ [\alpha, [\alpha, \beta]] = 1 = [\beta, [\alpha, \beta]] \right \rangle,$$ where $\alpha$ (resp. $\beta$) corresponds to the generator $(0, 1, 0)$ (resp. $(0, 0, 1)$) \cite{Kahn}*{Proposition 3}. Note that $H_3 (\mathbb Z)$ is a finitely-generated torsion-free nilpotent group ($2$-step) having the property of polynomial growth with growth exponent $4.$ The associated lower central series is given by $$\Gamma_1 \supseteq \Gamma_2 \supseteq \Gamma_3,$$ where $\Gamma_1 = H_3 (\mathbb Z), \Gamma_2 = \left [H_3 (\mathbb Z), H_3 (\mathbb Z) \right ] = \left \langle [\alpha, \beta] \right \rangle \cong \mathbb Z$ and $\Gamma_3 = \left [H_3 (\mathbb Z), \left \langle [\alpha, \beta] \right \rangle \right ] = \{1\}.$ Consider the automorphism $\phi : H_3 (\mathbb Z) \rightarrow H_3 (\mathbb Z)$ given by $\phi (\alpha) = (0, 2, 1)$ and $\phi (\beta) = (0, 1, 1).$ Following the notation introduced earlier, $\phi_1$ is the induced automorphism on the abelianization $H_3 (\mathbb Z)_{\mathrm {ab}}.$ Identifying $H_3 (\mathbb Z)_{\mathrm {ab}}$ with $\mathbb Z^{2}$ in the obvious way, it turns out that the matrix of $\phi_1$ is $\begin{pmatrix} 2 & 1 \\ 1 & 1 \end{pmatrix} \in GL_2 (\mathbb Z).$ Again by the same notation, $\phi_{2}$ is the induced automorphism of $\mathbb{Z}$ which is multiplication by $1$. Therefore, $\mathrm{max}\{h_{\mathrm{alg}}(\phi_1),h_{\mathrm{alg}}(\phi_2)\}= \log \left (\frac{3+\sqrt{5}}{2} \right )$. Also by the same notation, $\rho_1=\frac{3+\sqrt{5}}{2}$ as the spectral radius of the matrix $\begin{pmatrix} 2 & 1 \\ 1 & 1 \end{pmatrix}$ is $\frac {3 + \sqrt {5}} {2}$. So, combining the Main Theorem with the equations \eqref{abelianization} and \eqref{algentlow}, we get the following:
\begin{prop}
    Let $H_3 (\mathbb Z)$ be the discrete Heisenberg group and $\phi$ be the natural automorphism of $H_3 (\mathbb Z)$ induced by the matrix of Arnold's cat map. Then
    \begin{displaymath}
     \log \left (\frac {3 + \sqrt {5}} {2} \right ) \leq \mathrm {Entp} (\phi) \leq 4 \log \left (\frac {3 + \sqrt {5}} {2} \right ).   
    \end{displaymath}
\end{prop}

\begin{rem}
    
    \begin{itemize}
        \item Our study suggests that the notions of metric dimension and product entropy introduced by Kerr behave well in the context of group $C^{\ast}$-algebras when the underlying group has polynomial growth. We have seen that for groups with exponential growth, the metric dimension of the associated group $C^{\ast}$-algebras is always $+\infty$.
        \vspace{2mm}
        \item We conjecture that if $\Gamma$ is a finitely generated discrete group with exponential growth, then the product entropy of any automorphism of $\Gamma$ is $+\infty$. Our conjecture will be true if one has a positive answer to the same conjecture about the algebraic entropy of automorphisms of finitely generated groups with exponential growth (see \cite{Dikranjan2013TopologicalEA}*{conjecture 5.3.15}). However it is well known that for any group $\Gamma,$ the inner automorphisms have the same growth type as that of $\Gamma$ (i.e., of $\mathrm {id}_{\Gamma}$) and they have the same algebraic entropy \cite{} So exponential growth of $\Gamma$ ensures infinite algebraic entropy of inner automorphisms \cite{Dikranjan2013TopologicalEA}*{Proposition 5.3.13.} and hence infinite product entropy by the virtue of Theorem \ref{lowerbound}. 

    \vspace{2mm}
        \item We conjecture that the upper bound of the Main theorem is optimal i.e. there is an automorphism $\phi$  of some finitely generated discrete group of polynomial growth such that $\mathrm{Entp}(\phi)$ attains the upper bound.
        \end{itemize}
\end{rem}
\appendix
\section{Higher order Leibnitz property}
\begin{lem} \label{leibnitz1}
For any $k \in \mathbb N$ we have \Bea \Delta_{\ell}^{k} (f_{1}f_{2}) = \sum\limits_{j = 0}^{k} \binom {k} {j} \Delta_{\ell}^{k - j} (f_{1}) \Delta_{\ell}^{j} (f_{2}), \Eea for all $f_{1},f_{2}\in C_{c}(\Gamma),$ where $\Delta_{\ell}^{0} = \text {id}$ and $\Delta_{\ell}^{n} (f)$ denotes the $n$-fold commutator of $M_{\ell}$ with $\lambda(f)$.
\end{lem}
\begin{proof}
Observe that for any $f\in C_{c}(\Gamma)$, $\lambda(f)$ maps $C_{c}(\Gamma)$ into $C_{c}(\Gamma)$. Therefore, the following algebraic manipulation can be done on the dense subspace $C_{c}(\Gamma)$ of $\ell^{2}(\Gamma)$ and subsequently extended to all of $\ell^{2}(\Gamma)$. We start with the easy to prove equality, which proves the lemma for $k=1$: 
\begin{displaymath} 
\Delta_{\ell}^{1} (f_{1}f_{2}) = \Delta_{\ell}^{1} (f_{1}) \lambda(f_{2}) + \lambda(f_{1}) \Delta_{\ell}^{1} (f_{2}).\end{displaymath}
Now suppose the result holds for $k,$ for some $k \geq 1.$  \Bea
&&\Delta_{\ell}^{k + 1} (f_{1}f_{2})\\ & = & \left [M_{\ell}, \Delta_{\ell}^{k} (f_{1}f_{2}) \right ] \\ & = & \sum\limits_{j = 0}^{k} \binom {k} {j} \left [M_{\ell}, \Delta_{\ell}^{k - j} (f_{1}) \Delta_{\ell}^{j} (f_{2}) \right ] \\ & = & \sum\limits_{j = 0}^{k} \binom {k} {j} \left [M_{\ell}, \Delta_{\ell}^{k-j} (f_{1}) \right ] \Delta_{\ell}^{j} (f_{2}) + \sum\limits_{j = 0}^{k} \binom {k} {j} \Delta_{\ell}^{k - j} (f_{1}) \left [M_{\ell}, \Delta_{\ell}^{j} (f_{2}) \right ] \\ & = & \sum\limits_{j = 0}^{k} \binom {k} {j} \Delta_{\ell}^{k-j+1} (f_{1}) \Delta_{\ell}^{j} (f_{2}) + \sum\limits_{j = 0}^{k} \binom {k} {j} \Delta_{\ell}^{k - j} (f_{1}) \Delta_{\ell}^{j+1} (f_{2}) \\ & = & \sum\limits_{j = 0}^{k} \binom {k} {j} \Delta_{\ell}^{k-j+1} (\lambda(f_{1})) \Delta_{\ell}^{j} (f_{2}) + \sum\limits_{j = 1}^{k + 1} \binom {k} {j - 1} \Delta_{\ell}^{k - j + 1} (f_{1}) \Delta_{\ell}^{j} (f_{2}) \\ & = & \Delta_{\ell}^{k + 1} (f_{1}) \lambda(f_{2}) + \sum\limits_{j = 1}^{k} \left (\binom {k} {j} + \binom {k} {j - 1} \right ) \Delta_{\ell}^{k - j + 1} (f_{1}) \Delta_{\ell}^{j} (f_{2}) + \lambda(f_{1}) \Delta_{\ell}^{k + 1} (f_{2}) \\ & = & \Delta_{\ell}^{k + 1} (f_{1}) \lambda(f_{2}) + \sum\limits_{j = 1}^{k} \binom {k + 1} {j} \Delta_{\ell}^{k + 1 - j} (f_{1}) \Delta_{\ell}^{j} (f_{2}) + \lambda(f_{1}) \Delta_{\ell}^{k + 1} (f_{2}) \\ & = & \sum\limits_{j = 0}^{k + 1} \binom {k + 1} {j} \Delta_{\ell}^{k + 1 - j} (f_{1}) \Delta_{\ell}^{j} (f_{2})
\Eea
This proves the result for $k + 1.$ Hence the result follows.\end{proof}
{\textbf Proof of Lemma \ref{higherleibnitz}}: We prove the lemma using two-variable induction. For $k, n \in \mathbb N,$ let the statement be denoted by $P (k, n).$ The statement $P (1, 1)$ holds trivially. Now take any $n > 1$ and assume that $P (1, k)$ holds for all $k < n.$ We have to first show that $P (1, n)$ holds. Now using Lemma \ref{leibnitz1} for $k = 1$ we have \Bea \Delta_{\ell}^{1} \left (x_{1} x_{2} \cdots x_{n} \right ) & = & \Delta_{\ell}^{1} \left (x_{1} x_{2} \cdots x_{n - 1} \right ) x_{n} + x_{1} x_{2} \cdots x_{n - 1} \Delta_{\ell}^{1} \left (x_{n} \right ) \\ & = & \sum\limits_{r_{1} + r_{2} + \cdots + r_{n} = 1} \prod\limits_{i = 1}^{n} \Delta_{\ell}^{r_{i}} \left (x_{i} \right ).
\Eea
This proves that $P(1, n)$ holds for any $n \in \mathbb N.$ Next we assume that $P (m, n)$ holds for any $m < k$ with $k > 1$ and for all $n \in \mathbb N.$
To finish the induction step we need to show that $P (k, n)$ holds. 
Then we have 
\Bea && \Delta_{\ell}^{k} \left (x_{1} x_{2} \cdots x_{n} \right )
\\ 
& = & \left [M_{\ell}, \Delta_{\ell}^{k - 1} \left (x_{1} x_{2} \cdots x_{n} \right ) \right ] \\ 
& = & \sum\limits_{r_{1} + r_{2} + \cdots + r_{n} = k - 1} \dbinom {k - 1} {r_{1}\ r_{2}\ \cdots\ r_{n}} \left [M_{\ell}, \prod\limits_{i = 1}^{n} \Delta_{\ell}^{r_{i}} \left (x_{i} \right ) \right ] \\ & = & \sum\limits_{r_{1} + r_{2} + \cdots + r_{n} = k - 1} \dbinom {k - 1} {r_{1}\ r_{2}\ \cdots\ r_{n}} \Delta_{\ell}^{1} \left (\prod\limits_{i = 1}^{n} \Delta_{\ell}^{r_{i}} \left (x_{i} \right ) \right ) \\ & = & \sum\limits_{r_{1} + r_{2} + \cdots + r_{n} = k - 1} \dbinom {k - 1} {r_{1}\ r_{2}\ \cdots\ r_{n}} \sum\limits_{t_{1} + t_{2} + \cdots + t_{n}} \prod\limits_{i = 1}^{n} \Delta_{\ell}^{r_{i} + t_{i}} \left (x_{i} \right ) \\ 
& = & \sum_{s_{1} + s_{2} + \cdots + s_{n} = k}
\left (
\begin{aligned}[t]
&\dbinom{k-1}{s_{1}-1\ s_{2}\ \cdots\ s_{n}}
 + \dbinom{k-1}{s_{1}\ s_{2}-1\ \cdots\ s_{n}}  \\
&+ \cdots
+ \dbinom{k-1}{s_{1}\ s_{2}\ \cdots\ s_{n}-1}
\end{aligned}
\right )
\prod_{i=1}^{n} \Delta_{\ell}^{s_{i}} (x_{i}) \\ & = & \sum\limits_{s_{1} + s_{2} + \cdots + s_{n} = k} \dbinom {k} {s_{1}\ s_{2}\ \cdots\ s_{n}} \prod\limits_{i = 1}^{n} \Delta_{\ell}^{s_{i}} \left (x_{i} \right ).
\Eea
Thus $P (k, n)$ holds and hence the result follows.

\vspace{5mm}

{\bf Acknowledgement}: The first author acknowledges the financial support under the Senior Research Fellowship Scheme funded by UGC. The authors are grateful to Dr. Shubhabrata Das for several helpful discussions on geometric group theory.

\bibliographystyle{amsplain}
\bibliography{References}
   
\end{document}